\documentclass[a4paper,10pt,reqno,oneside]{amsart}
\usepackage{amsmath,amsthm,amssymb,enumerate,mathtools,stmaryrd,diagbox}
\usepackage{cancel,soul}
\usepackage[latin1]{inputenc}
\usepackage{bbm}
\usepackage{esint}
\usepackage{multirow}
\usepackage{libertine}
\usepackage{euscript,mathrsfs}
\usepackage{xcolor}
\usepackage[left=3.5cm,right=3.5cm,top=3.0cm,bottom=2.5cm]{geometry}
\usepackage[colorlinks=true, linktocpage=true, linkcolor=red!70!black, citecolor=green!50!black, urlcolor=black]{hyperref}
\usepackage{breakurl}
\usepackage{colortbl}

\allowdisplaybreaks
\usepackage{tikz}
\usepackage{scalerel}

\usepackage[colorinlistoftodos,prependcaption,textsize=tiny]{todonotes}

\usepackage{enumitem}
\setenumerate{label={\rm (\alph{*})}}

\numberwithin{equation}{section}

\newcommand{\bbone}{\text{\usefont{U}{bbold}{m}{n}1}}
\MakeRobust{\bbone}
\newcommand{\R}{\mathbb R}
\newcommand{\C}{\mathbb C}
\newcommand{\N}{\mathbb N}

\newcommand{\D}{\mathrm{D}}
\DeclareMathOperator{\Curl}{Curl}
\DeclareMathOperator{\Div}{Div}

\newcommand{\dev}{\operatorname{dev}}
\newcommand{\sym}{\operatorname{sym}}
\renewcommand{\skew}{\operatorname{skew}}

\DeclareMathOperator{\tr}{tr}
\newcommand{\inc}{\boldsymbol{\operatorname{inc}}\,}

\newcommand{\lebe}{\operatorname{L}}
\newcommand{\sobo}{\operatorname{W}}
\newcommand{\hold}{\operatorname{C}}
\newcommand{\besov}{\operatorname{B}}

\newcommand{\norm}[1]{\left\lVert#1\right\rVert}
\newcommand{\skalarProd}[2]{\big\langle#1,#2\big\rangle}

\newcommand{\abs}[1]{\lvert#1\rvert}
\newcommand{\Anti}{\operatorname{Anti}}
\newcommand{\E}{\mathbf{E}}

\theoremstyle{plain}
\newtheorem{theorem}{Theorem}[section]
\newtheorem{lemma}[theorem]{Lemma}

\newtheorem{corollary}[theorem]{Corollary}

\theoremstyle{remark}
\newtheorem{remark}[theorem]{Remark}
\newtheorem{example}[theorem]{Example}

\begin{document}
\numberwithin{equation}{section}

\colorlet{RED}{red}%%%
\title[Generalised $\lebe^1$-KMS Inequalities]{Limiting Korn-Maxwell-Sobolev inequalities\\ for general incompatibilities
}
\author[F. Gmeineder]{Franz Gmeineder}
\address[Franz Gmeineder]{Department of Mathematics and Statistics, University of Konstanz, Universit\"{a}tsstrasse 10, 78457 Konstanz, Germany.}
\email{franz.gmeineder@uni-konstanz.de}
\author[P. Lewintan]{Peter Lewintan}
\address[Peter Lewintan]{Karlsruhe Institute of Technology, Englerstrasse 2, 76131 Karlsruhe, Germany}
\email{peter.lewintan@kit.edu}
\author[J. Van Schaftingen]{Jean Van Schaftingen}
\address[Jean Van Schaftingen]{Universit\'{e} catholique de Louvain, Institut de Recherche en Math\'{e}matique et Physique, Chemin du Cyclotron 2 bte L7.01.02, 1348 Louvain-la-Neuve, Belgium. }
\email{jean.vanschaftingen@uclouvain.be}

\date\today
\maketitle
\begin{abstract} 
We give sharp conditions for the limiting Korn-Maxwell-Sobolev inequalities 
\begin{align*}
  \norm{P}_{{\dot{\sobo}}{^{k-1,\frac{n}{n-1}}}(\R^n)}\le c\big(\norm{\mathscr{A}[P]}_{{\dot{\sobo}}{^{k-1,\frac{n}{n-1}}}(\R^n)}+\norm{\mathbb{B}P}_{\lebe^{1}(\R^n)}\big)
 \end{align*}
to hold for all $P\in\hold_{c}^{\infty}(\R^{n};V)$, where $\mathscr{A}$ is a linear map between finite dimensional vector spaces and $\mathbb{B}$ is a $k$-th order, linear and homogeneous constant-coefficient differential operator. By the appearance of the $\lebe^{1}$-norm of the differential expression $\mathbb{B}P$ on the right-hand side, such inequalities generalise previously known estimates to the borderline case $p=1$, and thereby answer an open problem due to \textsc{M\"{u}ller, Neff} and the second author (Calc\@. Var\@. PDE, 2021) in the affirmative.

\end{abstract}
\section{Introduction}\label{sec:introduction}
\subsection{KMS-type inequalities}
Korn-type inequalities are a central tool in (non-)linear elasticity or fluid mechanics; see, e.g. \cite{Friedrichs,Gobert,Korn}. In their easiest form, they allow to control the full gradient of some $u\in\hold_{c}^{\infty}(\R^{n};\R^{n})$ by its symmetric or trace-free symmetric parts, respectively: For each $n\in\mathbb{N}$ and $1<q<\infty$, there exists a constant $c=c(n,q)>0$ such that 
\begin{align}\label{eq:KornIntro1}
\begin{split}
&\norm{\D u}_{\lebe^{q}(\R^{n})}\leq c\norm{\sym\D u}_{\lebe^{q}(\R^{n})}, \\ 
&\norm{\D u}_{\lebe^{q}(\R^{n})} \leq c\norm{\sym \D u - \tfrac{1}{n}\mathrm{div}(u)\, \bbone_{n}}_{\lebe^{q}(\R^{n})}
\end{split}
\end{align}
hold for all $u\in\hold_{c}^{\infty}(\R^{n};\R^{n})$, where $\bbone_{n}$ is the $(n\times n)$-unit matrix. Inequalities of the form \eqref{eq:KornIntro1} are non-trivial because the non-differential counterparts 
\begin{align}\label{eq:KornIntro2}
\begin{split}
&\norm{P}_{\lebe^{q}(\R^{n})}\leq c\norm{\sym P}_{\lebe^{q}(\R^{n})},\;\;\;\norm{P}_{\lebe^{q}(\R^{n})} \leq c\norm{\sym P - \tfrac{1}{n}\tr(P)\,\bbone_{n}}_{\lebe^{q}(\R^{n})}
\end{split}
\end{align}
for general fields $P\in\hold_{c}^{\infty}(\R^{n};\R^{n\times n})$ are trivially seen to be false; take, e.g., fields $P$ with values in the skew-symmetric matrices with zero trace. In view of the failure of \eqref{eq:KornIntro2}, it is crucial for \eqref{eq:KornIntro1} that we deal with \emph{gradients} here. In this case, it is now well-known that inequalities \eqref{eq:KornIntro1} are a consequence of Calder\'{o}n-Zygmund estimates, and thus they are bound to only hold in the regime $1<q<\infty$. 

We may rephrase this discussion by saying that despite the failure of  \eqref{eq:KornIntro1} for general matrix fields $P$, it does hold indeed for $\Curl $-free matrix fields. One might thus hope to modify \eqref{eq:KornIntro2} in a way such that \eqref{eq:KornIntro1} is retrieved for $\Curl $-free fields, meaning that we ask for an inequality of the form 
\begin{align}\label{eq:KMSMain1}
\norm{P}_{\lebe^{q}(\R^{n})}\leq c\Big(\norm{\mathscr{A}[P]}_{\lebe^{q}(\R^{n})} + \norm{\Curl P}_{\lebe^{p}(\R^{n})}\Big),\qquad P\in\hold_{c}^{\infty}(\R^{n};\R^{n\times n}), 
\end{align}
where $\mathscr{A}[P]=\sym P$ or $\mathscr{A}[P]=\sym P - \frac{1}{n}\tr(P)\bbone_{n}$, respectively. Note that, if $1\leq p<n$ is fixed, then scaling uniquely determines $q$ as $q=\frac{np}{n-p}$. For the choices of $\mathscr{A}$ displayed above and depending on $n$ and $1\leq p<n$, inequality \eqref{eq:KMSMain1} has been established in \cite{ContiGarroni,GLP,GLN1,GmSp}, and we refer the reader to the subsequent section for a discussion of the underlying strategies of proof.

Numerous applications, e.g. from continuum mechanics (cf. \cite{Arnold,BauerNeffPaulyStarke,Nerf,NeffPaulyWitsch,LMN}), require to go beyond \eqref{eq:KMSMain1} in two different directions. First, one strives for more general part maps $\mathscr{A}$ than the (trace-free) symmetric parts. Second, one aims for more general incompatibilities than given by the curl of a matrix field. Such generalisations prove particularly relevant e.g. in the study of tensor fields with conformally invariant dislocation energies \cite{LMN} or the relaxed micromorphic model \cite{Nerf}. In this vein, one aims for sharp conditions on the interplay between the part maps, incompatibilities and integrabilities $p$ which make inequalities \eqref{eq:KMSMain1} work. As we shall discuss now, the only remaining case which could not be treated successfully so far is the case $p=1$, and the aim of the present paper is to close this gap.

\subsection{Context and main results}
Let $1\leq p <n$. Based on our above discussion, the general form of the requisite inequalities reads as
 \begin{align}\label{eq:KMS-I}
  \norm{P}_{{\dot{\sobo}}{^{k-1,q}}(\R^n)}\le c\left(\norm{\mathscr{A}[P]}_{{\dot{\sobo}}{^{k-1,q}}(\R^n)}+\norm{\mathbb{B}P}_{\lebe^p(\R^n)}\right), \qquad  P\in\hold^\infty_c(\R^n;V), \tag{KMS}
 \end{align}
where $V,\widetilde V, W$ are finite dimensional vector spaces, $\mathscr{A}: V\to\widetilde V$ is a linear map and $\mathbb{B}$ is a linear, homogeneous, $k$-th order constant coefficient differential operator on $\R^n$ from $V$ to $W$. In the sequel, we shall refer to \eqref{eq:KMS-I} as \emph{Korn-Maxwell-Sobolev inequalities}, and we recall that scaling determines $q=p^*=\frac{np}{n-p}$ as the Sobolev exponent of $1\leq p<n$. In order to single out the borderline case to be addressed in this paper, we briefly summarise the available results concerning inequalities \eqref{eq:KMS-I}. Here we make use of the Fourier symbol terminology which, for the reader's convenience, is displayed in Section \ref{sec:notation} below.

\emph{The case $1<p<n$.} In this case, there is a vast literature on various constellations of $\mathscr{A},\mathbb{B}$ and $p$, see e.g. the references in \cite{GLN2,LMN}. If $1<p<n$ and $\mathbb{B}=\Curl $, \eqref{eq:KMSMain1} can be established by performing a Helmholtz decomposition on $P$ and subsequently estimating the divergence-free part by use of the fractional integration theorem and the curl-free part by virtue of the usual Korn-type inequalities, cf. \cite{GLN1,GmSp}. This approach is difficult to be implemented in the case of general incompatibility operators $\mathbb{B}$. Based on the so-called \emph{algebraic split} approach, a complete characterisation of the interplay between the part map $\mathscr{A}$ and the differential operator $\mathbb{B}$ in Korn-Maxwell-Sobolev inequalities has been given in the precursor \cite{GLN2}. The outcome is
\begin{align}\label{eq:1<p<n}
\eqref{eq:KMS-I}\;\text{holds for}\;1<p<n \Longleftrightarrow \bigcup_{\xi\neq 0}\ker\mathscr{A}\cap\ker\mathbb{B}[\xi]=\{0\}, 
\end{align}
which extends and unifies several previous results, cf. \cite{BauerNeffPaulyStarke,LMN,Lewintan3,Lewintan4,NeffPlastic,NeffPaulyWitsch}. 
In Fourier analytic terms, the right-hand side condition of \eqref{eq:1<p<n} means that $\mathbb{B}$ behaves like an elliptic operator on $\ker\mathscr{A}$ (see Section \ref{sec:definitions} for this notion). Hence, \eqref{eq:KMS-I} embodies two principles: First, applying \eqref{eq:1<p<n} to fields $P\in\hold_{c}^{\infty}(\R^{n};\ker\mathscr{A})$, gives us the Sobolev-type inequality 
\begin{align}\label{eq:Sobolevintro}
\norm{P}_{{\dot{\sobo}}{^{k-1,\frac{np}{n-p}}}(\R^n)}\leq c\norm{\mathbb{B}P}_{\lebe^{p}(\R^{n})}. 
\end{align}
Second, if $k=1$ and $\mathbb{B}=\Curl$, the nullspace of $\mathbb{B}$ precisely consist of gradient fields. Thus, applying \eqref{eq:KMS-I} to fields $P=\D u$, we then retrieve the Korn-type inequality 
\begin{align}\label{eq:Kornintro}
\norm{\D u}_{\lebe^{\frac{np}{n-p}}(\R^n)}\leq c\norm{\mathscr{A}[\D u]}_{\lebe^{\frac{np}{n-p}}(\R^n)}. 
\end{align}
In this sense, \eqref{eq:KMS-I} can be seen as a common gateway to both Korn and Sobolev inequalities.

\emph{The case $p=1$.} On the contrary, only few results, namely \cite{ContiGarroni,GLP,GLN1,GmSp}, address the borderline case $p=1$ (whereby $q=1^{*}\coloneqq\frac{n}{n-1}$) and only in the particular situation of the differential operator $\mathbb{B}=\Curl$ being the matrix curl. Specifically, in this case the validity of \eqref{eq:KMS-I} is equivalent to 
\begin{itemize}
 \item $\R$-ellipticity of $\mathbb{A}u\coloneqq\mathscr{A}[\D u]$ in $n\geq3$ dimensions (cf. \cite{GmSp,GLN1}),
 \item $\C$-ellipticity of $\mathbb{A}u\coloneqq\mathscr{A}[\D u]$ in $n=2$  dimensions (cf. \cite{GLN1}).
\end{itemize}
In essence, the stronger condition of $\mathbb{C}$-ellipticity compensates the weaker properties of the operator $\Curl$ in two dimensions. The methods of \cite{GmSp,GLN1}, however, are very specific to the $\Curl $-operator, and do not allow for an ad-hoc generalisation to the case of general part maps $\mathscr{A}$ and operators $\mathbb{B}$. In particular, it is far from clear how they admit conclusive statements on the validity of \eqref{eq:KMS-I} when $\Curl $ is replaced by e.g.  $\mathbb{B}=\dev\sym\Curl$. Even for this operator, which takes a prominent role in gradient plasticity with plastic spin or incompatible elasticity (cf. \textsc{M\"{u}ller} et al. \cite{LMN}), inequalities \eqref{eq:KMS-I} thus have remained an open problem. 

In the borderline case $p=1$, inequalities of the form \eqref{eq:KMS-I} still imply Korn and Sobolev inequalities \eqref{eq:Sobolevintro}, \eqref{eq:Kornintro}. Specifically, going to \eqref{eq:Sobolevintro}, it is clear that $\mathbb{B}$ must match the conditions which make limiting Sobolev-type estimates work for $p=1$ -- at least, when acting on fields $P\in\hold_{c}^{\infty}(\R^{n};\ker\mathscr{A})$. These conditions, due to the third author \cite{VS} (also see \textsc{Bourgain \& Brezis} \cite{BB}), require $\mathbb{B}$ to behave like an elliptic and \emph{cancelling} operator when restricted to fields $P\in\hold_{c}^{\infty}(\R^{n};\ker\mathscr{A})$. Note that, by the failure of Calder\'{o}n-Zygmund  estimates for $p=1$ (cf. \textsc{Ornstein} \cite{Ornstein}), \eqref{eq:Sobolevintro} cannot be derived from the ellipticity assumption on the right-hand side of \eqref{eq:1<p<n} alone. The sharp conditions for \eqref{eq:KMS-I} to hold thus must incorporate an additional cancellation-type condition. The necessity of such an additional condition can directly be seen by the following explicit
\begin{example}\label{ex:noninequality}
In $n=3$ dimensions consider the part map $\mathscr{A}=\sym:\R^{3\times3}\to\R^{3\times3}_{\sym}$, $X\mapsto\frac12(X+X^\top)$ and the differential operator $\mathbb{B}:\hold^\infty_c(\R^3;\R^{3\times3})\to\hold^\infty_c(\R^3;\R^{3\times3})$ given by $\mathbb{B}P\coloneqq \skew\Curl P + \tr \Curl P\cdot\bbone_3$. In the precursor \cite{GLN2} we have shown that for all  $P\in\hold^\infty_c(\R^3;\R^{3\times3})$ and all $1<p<3$ it holds
 \begin{align}\label{eq:symskewtr}
  \norm{P}_{\lebe^{p^*}(\R^3)}\le c\left(\norm{\sym P}_{\lebe^{p^*}(\R^3)}+\norm{\skew\Curl P}_{\lebe^p(\R^3)}+ \norm{\tr\Curl P}_{\lebe^p(\R^3)}\right).
 \end{align}
However, the following example shows that this estimate does not persist in the borderline case $p=1$. To this end, let us consider for $\varphi\in\hold^\infty_c(\R^3)$ the matrix field
\begin{subequations}
\begin{equation}
 P_{\varphi}=\begin{pmatrix} 0 & -\partial_3\varphi & \partial_2\varphi \\ \partial_3\varphi & 0 & -\partial_1\varphi \\ -\partial_2\varphi & \partial_1\varphi & 0 \end{pmatrix}.
\end{equation}
Then we have
\begin{equation}
 \Curl P_\varphi = \Delta\varphi\cdot\bbone_3-\D\nabla\varphi \in\R^{3\times3}_{\sym},
\end{equation}
so that
\begin{equation}
 \mathscr{A}[P_{\varphi}]=0 \quad \text{and}\quad \mathbb{B}P_{\varphi}=2\Delta\varphi\cdot\bbone_3.
\end{equation}
\end{subequations}
If \eqref{eq:symskewtr} were correct in the borderline case $p=1$ we would have
\begin{align*}
 \norm{\nabla\varphi}_{\lebe^{\frac32}(\R^3)} &= c\norm{P_{\varphi}}_{\lebe^{\frac32}(\R^3)}\overset{!}{\le} c\left(\norm{\mathscr{A}[P_\varphi]}_{\lebe^{\frac32}(\R^3)}+\norm{\mathbb{B}P_\varphi}_{\lebe^1(\R^3)}\right)=c\norm{\Delta\varphi}_{\lebe^1(\R^3)},
\end{align*}
and this inequality is easily seen to be false by taking regularisations and smooth cut-offs of the fundamental solution $\Phi(\cdot)=\frac{1}{3\omega_3}\frac{1}{\abs{\cdot}}$ of the Laplacian $(-\Delta)$; here, $\omega_{3}\coloneqq\mathscr{L}^{3}(\mathrm{B}_{1}(0))$.  
\end{example}
As our main result, to be stated as Theorem \ref{thm:L1KMS} in Section \ref{sec:main} below, the previous example can be classified within the \emph{sharp} conditions which make inequalities \eqref{eq:KMS-I} work. Namely, we have that \eqref{eq:KMS-I} holds for $p=1$ if and only if 
 \begin{align}\label{eq:KMSsharpconds}
 \bigcup_{\xi\in\R^{n}\setminus\{0\}} \ker\mathscr{A}\cap \ker\mathbb{B}[\xi] =\{0\} \;\;\;\text{and}\;\;\;
   \bigcap_{\xi\in\R^{n}\setminus\{0\}}\mathbb{B}[\xi](\ker\mathscr{A})=\{0\}
\end{align}
The second part of \eqref{eq:KMSsharpconds} is in general substantially weaker than the full cancellation condition $\bigcap_{\xi\neq 0}\mathbb{B}[\xi](V)=\{0\}$. Specifically, the operator $\dev\sym\Curl$ serves as an example of an operator which is \emph{not cancelling} but satisfies \eqref{eq:KMSsharpconds} for basically all choices of part maps $\mathscr{A}$ which are relevant in applications; see Section \ref{sec:examples} for this and more examples. 
\subsection{Structure of the paper}
Besides this introduction, the paper is organised as follows: In Section \ref{sec:notation} we fix notation and collect the requisite background terminology and facts on differential operators as required in the sequel. We then state and prove our main Theorem \ref{thm:L1KMS} as well as variants for other function spaces, and discuss their relations to strong \textsc{Bourgain-Brezis} estimates in Section \ref{sec:main}. The paper is then concluded in Section \ref{sec:examples} by examples completing the picture of available limiting KMS-inequalities. Whereas we focus on inequalities on full space, our results still allow to provide an affirmative answer to the borderline case of the $\dev\sym\Curl$-operator left open in  \cite[Thm.\@ 3.5, Rem.\@ 3.6]{LMN} for \emph{globally vanishing traces} -- see Corollary \ref{cor:affirmative} and Example \ref{ex:devsymCurl}.

\section{Notation and preliminaries}\label{sec:notation}
\subsection{General notation} We will denote by $\skalarProd{\cdot}{\cdot}$ the inner product of a real finite dimensional space. For a square matrix $X\in\R^{n\times n}$ we consider the following algebraic parts: the transpose $X^\top$, the trace $\tr X\coloneqq \skalarProd{X}{\bbone_n}$, the deviatoric or trace-free part $\dev X \coloneqq X-\frac{\tr X}{n}\,\bbone_n$, the symmetric part $\sym X \coloneqq \frac12(X+X^\top)$ and the skew-symmetric part $\skew X \coloneqq\frac12(X-X^\top)$.

Recall that the matrix $\Curl$ is the row-by-row application of the vectorial curl, and therefore $(\Curl P)_{ijk}=\partial_{i}P_{kj}-\partial_{j}P_{ki}$ for $(m\times n)$-valued maps $P$.  In order to study generalised incompatibilities we recall some general terminology for vectorial differential operators. 

\subsection{Differential operators}\label{sec:definitions}
 Given $n,\ell\in\N$ and a finite dimensional real vector space $V$,  consider a homogeneous, linear, constant coefficient differential operator  $\mathbb{A}$ of order $\ell$ on $\R^n$ from $V$ to another finite dimensional vector space $W$. This means that we have a representation
 \begin{subequations}\label{eq:diffoperator}
\begin{align}
 \mathbb{A}v\coloneqq\sum_{\abs{\alpha}=\ell}\mathbb{A}_{\alpha}\partial^\alpha v, \quad v:\R^n\to V,
\end{align}
with linear maps $\mathbb{A}_\alpha:V\to W$ and multi-indices $\alpha\in\N^n_0$, $|\alpha|=\ell$. The corresponding \emph{symbol map} reads
\begin{align}
 \mathbb{A}[\xi]:V\to W, \quad \mathbb{A}[\xi]\boldsymbol{v}\coloneqq\sum_{\abs{\alpha}=\ell}\xi^\alpha\mathbb{A}_{\alpha}\boldsymbol{v}, \quad \xi\in\R^n, \boldsymbol{v}\in V,
\end{align}
\end{subequations}
where $\xi^\alpha\coloneqq\xi_1^{\alpha_1}\cdots\xi_n^{\alpha_n}$ for $\alpha=(\alpha_1 \, \cdots \, \alpha_n )^\top\in\N^n_0$. The operator $\mathbb{A}$ is then called
\begin{itemize}
\item ($\R$-)\emph{elliptic} if for every $\xi\in\R^n\backslash\{0\}$ the associated symbol map $\mathbb{A}[\xi]$ is injective, i.e.,
\begin{equation}
 \ker_{\R}\mathbb{A}[\xi]=\{0\} \quad \text{for all } \xi\in\R^n\backslash\{0\}.
\end{equation}
\item \emph{$\C$-elliptic} if for every $\xi\in\C^n\setminus\{0\}$ the complex symbol map $\mathbb{A}[\xi]:V+\mathrm{i}V\to W+ \mathrm{i}W$ is injective.
\item \emph{cancelling} if 
\begin{equation}
 \bigcap_{\xi\in\R^n\backslash\{0\}} \mathbb{A}[\xi](V) = \{0\}.
\end{equation}
\end{itemize}
We now gather some basic examples; more elaborate ones, which also show the interplay between the above notions and the dimension $n\in\mathbb{N}$ can be found in Section \ref{sec:examples}: 

Classical $\C$-elliptic first order differential operators are the gradient, the deviatoric gradient and the symmetric gradient. The deviatoric symmetric gradient is precisely $\C$-elliptic in dimensions $n\ge3$ and in $n=2$ dimensions only $\R$-elliptic, cf. \cite[\S 2.2]{BDG}. The curl or the divergence operator are not elliptic.

The classical curl and the generalised curl in dimensions $n\ge3$ are cancelling operators. The generalised curl in $2$ dimensions corresponds to the divergence and is not cancelling. The divergence is also not cancelling in any dimension. The corresponding matrix differential operators (which act row-wise) possess the same properties, i.e., the matrix $\Curl$ is cancelling. 

It is precisely the additional cancellation property which is required in limiting Sobolev-type inequalities on full space, and which we recall for the reader's convenience:
\begin{lemma}[{\cite[Thm.\@ 1.3]{VS}}]
Let $\mathbb{A}$ be an operator of the form \eqref{eq:diffoperator}. Then one has the estimate 
\begin{align}\label{eq:Sobolev}
\norm{\D^{\ell-1}u}_{\lebe^{\frac{n}{n-1}}(\R^{n})}\leq c\norm{\mathbb{A}u}_{\lebe^{1}(\R^{n})},\qquad u\in\hold_{c}^{\infty}(\R^{n};V), 
\end{align}
if and only if $\mathbb{A}$ is elliptic and cancelling. 
\end{lemma}
Compared with ellipticity and cancellation, the notion of $\mathbb{C}$-ellipticity usually appears when aiming for sharp conditions for boundary estimates \cite{BDG,GmRaVa1} rather than for estimates on full space. However, since an operator is $\mathbb{C}$-elliptic if and only if its nullspace in the space of distributions $\mathscr{D}'(\R^{n};V)$ is finite-dimensional, it is usually easier to decide whether an operator is $\mathbb{C}$-elliptic than elliptic and cancelling. The following observation connects $\C$-ellipticity and ellipticity and cancellation in all dimensions:
\begin{lemma}[\cite{GmRa1,GmRaVa1}]\label{lem:Celliptic}
The $\C$-ellipticity of an operator of the form \eqref{eq:diffoperator}  implies both its ellipticity and cancellation. For $n=2$ also the converse implication holds true for first order operators. 
\end{lemma}

Furthermore, the differential operator $\mathbb{A}$ of the form \eqref{eq:diffoperator} is called \emph{cocancelling} if
\begin{equation}
 \bigcap_{\xi\in\R^n\backslash\{0\}}\ker\mathbb{A}[\xi]=\{0\}.
\end{equation}
The classical example of a cocancelling operator is the divergence, see \cite{VS}. It is precisely the cocancelling operators, which appear in strong \textsc{Bourgain}-\textsc{Brezis}-type estimates:

\begin{lemma}[{\cite{BB}, {\cite[Thm.\@ 9.2]{VS}}}]\label{lem:strongBBcocancelling}
Let $n\ge2$ and $\mathbb{A}$ be an $\ell$-th order, homogeneous, linear, constant coefficient, cocancelling differential operator on $\R^n$ from $V$ to $W$. Then there exists a constant $c=c(\mathbb{A})>0$ such that we have 
\begin{equation}\label{eq:strongBBcocancelling}
 \norm{f}_{\dot{\sobo}{}^{-1,1^*}(\R^n)}\le c\left(\norm{\mathbb{A}f}_{\dot{\sobo}{}^{-1-\ell,1^*}(\R^n)}+\norm{f}_{\lebe^1(\R^n)}\right)\qquad\text{for all}\;f\in\hold_{c}^{\infty}(\R^{n};V).
\end{equation}
\end{lemma}
\section{Limiting $\lebe^{1}$-KMS inequalities}\label{sec:main}
\subsection{The main result and its proof}
We now come to our main result. For the following, let $V,\widetilde V,W$ be finite dimensional real inner product spaces. We then have 
\begin{theorem}[Limiting $\lebe^1$-KMS-inequalities]\label{thm:L1KMS} Let $n\ge2$ and $k\in\mathbb{N}$. Moreover, let  $\mathscr{A}:V\to\widetilde V$ be a linear  map and $\mathbb{B}$ be a  $k$-th order, linear, homogeneous, constant coefficient differential operator on $\R^n$ from $V$ to $W$. Then the following are equivalent:
 \begin{enumerate}
  \item\label{item:L1KMSA1} There exists a constant $c=c(\mathscr{A},\mathbb{B})>0$ such that the inequality
   \begin{equation} \label{eq:thm:L1KMS1}
      \norm{P}_{{\dot{\sobo}}{^{k-1,1^*}}(\R^n)}\leq c\,\left(\norm{\mathscr{A}[P]}_{{\dot{\sobo}}{^{k-1,1^*}}(\R^n)}+\norm{\mathbb{B} P}_{\lebe^1(\R^{n})}\right) 
   \end{equation}
   holds for all $P\in\hold^\infty_c(\R^n;V)$.
   \item\label{item:L1KMSA2} $\mathbb{B}$ is \emph{reduced elliptic} and \emph{reduced cancelling} (relative to $\mathscr{A}$), meaning that 
 \begin{align}\label{eq:ellipticitycancellingreduced}
   \bigcup_{\xi\in\R^{n}\setminus\{0\}}\ker\mathscr{A}\cap \ker\mathbb{B}[\xi] =\{0\} \quad\text{and}\quad
   \bigcap_{\xi\in\R^{n}\setminus\{0\}}\mathbb{B}[\xi](\ker\mathscr{A})=\{0\}.
\end{align}
\end{enumerate}
\end{theorem}
Before we come to the proof of Theorem \ref{thm:L1KMS}, some comments are in order: 
\begin{remark}
 We do not impose the stronger condition that $\mathbb{B}$ is cancelling on the whole $V$, i.e., $\bigcap_{\xi\in\R^{n}\setminus\{0\}}\mathbb{B}[\xi](V)=\{0\}$. The latter implies reduced cancellation. Indeed, for all $V_1\subseteq V$ we have always
 \begin{equation*}
  \mathbb{B}[\xi](V_1)\subseteq\mathbb{B}[\xi](V).
 \end{equation*}
\end{remark}

\begin{remark}
 For $\mathbb{B}=\Curl$ and $n\ge3$ the condition \eqref{eq:ellipticitycancellingreduced}$_2$ is trivially satisfied since curl possesses the cancellation property. On the other hand, in two  dimensions the conditions \eqref{eq:ellipticitycancellingreduced} are equivalent to $\mathbb{B}$ being reduced $\C$-elliptic (relative to $\mathscr{A}$), see \cite{GmRaVa1}. Thus, for the particular choice $\mathbb{B}=\Curl$ we recover the results from the precursor \cite{GLN1}; for detailed discussion see Section \ref{sec:L1Curl}.
\end{remark}

By Lemma \ref{lem:Celliptic} we directly obtain the following consequence of Theorem \ref{thm:L1KMS}:

\begin{corollary}\label{cor:reducedCelliptic}
If a first order, linear, homogeneous differential operator $\mathbb{B}$  from $V$ to $W$ is reduced $\C$-elliptic (relative to $\mathscr{A}$), meaning that 
\begin{align*}
\mathbb{B}\colon\hold_{c}^{\infty}(\R^{n};\ker\mathscr{A})\to \hold_{c}^{\infty}(\R^{n};W)\quad\text{is $\mathbb{C}$-elliptic}, 
\end{align*}
then the estimate \eqref{eq:thm:L1KMS1} holds true.
\end{corollary}

\begin{proof}[Proof of Theorem \ref{thm:L1KMS}]
  We start with the direction  `\ref{item:L1KMSA2} $\Rightarrow$ \ref{item:L1KMSA1}', and hence let $P\in\hold_{c}^{\infty}(\R^{n};V)$ be given. As in \cite{GLN2}, we perform an algebraic split of the vector field $P$. More precisely, we denote by $\Pi_{\ker\mathscr{A}}$ and $\Pi_{(\ker\mathscr{A})^{\bot}}$ the orthogonal projections onto $\ker\mathscr{A}$ and $(\ker\mathscr{A})^{\perp}$, respectively. Given $\alpha\in\mathbb{N}_{0}^{n}$ with $|\alpha|=k-1$, we then write
\begin{align}\label{eq:orthdecomp}
  \partial^{\alpha}P = \Pi_{\ker\mathscr{A}}[\partial^{\alpha}P]+\Pi_{(\ker\mathscr{A})^\perp}[\partial^{\alpha}P]
\end{align}  
and treat both parts separately. Firstly, since $\mathscr{A}|_{\ker(\mathscr{A})^\perp}$ is injective, we have the pointwise estimate 
  \begin{equation}\label{eq:injectivity}
  	\begin{split}
   |\partial^{\alpha}\Pi_{(\ker\mathscr{A})^\perp}P(x)| & =|\Pi_{(\ker\mathscr{A})^\perp}\partial^{\alpha}P(x)| \\ & \le c\,|\mathscr{A}[\Pi_{(\ker\mathscr{A})^\perp}\partial^{\alpha}P(x)]|= c\,|\mathscr{A}[\partial^{\alpha}P(x)]|
   \end{split}
  \end{equation}
for all $x\in\R^{n}$, where $c=c(\mathscr{A})>0$ is a constant. Secondly, the reduced ellipticity from  $\eqref{eq:ellipticitycancellingreduced}_{1}$ yields that  $\mathbb{B}:\hold^\infty_c(\R^n;\ker\mathscr{A})\to \hold^\infty_c(\R^n;W)$ is elliptic, and this implies that there exists a constant $c>0$ such that 
\begin{align}\label{eq:eldorado}
\|\widetilde{P}\|_{\lebe^{1^{*}}(\R^{n})}\leq c\,\|\mathbb{B}\widetilde{P}\|_{{\dot\sobo}{^{-k,1^{*}}}(\R^{n})}\qquad\text{for all}\;\widetilde{P}\in\hold_{c}^{\infty}(\R^{n};\ker\mathscr{A}). 
\end{align}
Now let $\alpha\in\mathbb{N}_{0}^{n}$ be such that $|\alpha|=k-1$. We then record that 
\begin{align}\label{eq:littlehelper}
\norm{\partial^{\alpha}\Pi_{\ker\mathscr{A}}[P]}_{\lebe^{1^{*}}(\R^{n})} & = \norm{\Pi_{\ker\mathscr{A}}[\partial^{\alpha}P]}_{\lebe^{1^{*}}(\R^{n})} \notag\\ 
& \!\!\stackrel{\eqref{eq:eldorado}}{\leq} c\,\norm{\mathbb{B}(\Pi_{\ker\mathscr{A}}[\partial^{\alpha}P])}_{{\dot{\sobo}}{^{-k,1^{*}}}(\R^{n})} \notag\\ 
& \!\!\stackrel{\eqref{eq:orthdecomp}}{=} c\,\norm{\mathbb{B}(\partial^{\alpha}P-\Pi_{(\ker\mathscr{A})^{\bot}}[\partial^{\alpha}P])}_{{\dot{\sobo}}{^{-k,1^{*}}}(\R^{n})} \notag\\ 
& \leq c\,\norm{\partial^{\alpha}\mathbb{B}P}_{{\dot{\sobo}}{^{-k,1^{*}}}(\R^{n})} + c\,\norm{\mathbb{B}\Pi_{(\ker\mathscr{A})^{\bot}}[\partial^{\alpha}P]}_{{\dot{\sobo}}{^{-k,1^{*}}}(\R^{n})}\notag\\ 
& \leq  c\,\norm{\mathbb{B}P}_{{\dot{\sobo}}{^{-1,1^{*}}}(\R^{n})} + c\,\norm{\mathbb{B}\Pi_{(\ker\mathscr{A})^{\bot}}[\partial^{\alpha}P]}_{{\dot{\sobo}}{^{-k,1^{*}}}(\R^{n})}\notag\\ 
& \leq  c\,\norm{\mathbb{B}P}_{{\dot{\sobo}}{^{-1,1^{*}}}(\R^{n})} + c\,\norm{\Pi_{(\ker\mathscr{A})^{\bot}}[\partial^{\alpha}P]}_{\lebe^{1^{*}}(\R^{n})}\notag\\ 
& \!\stackrel{\eqref{eq:injectivity}}{\leq}  c\,\norm{\mathbb{B}P}_{{\dot{\sobo}}{^{-1,1^{*}}}(\R^{n})} + c\,\norm{\mathscr{A}[\partial^{\alpha}P]}_{\lebe^{1^{*}}(\R^{n})}\notag\\ 
& \leq c\,\norm{\mathbb{B}P}_{{\dot{\sobo}}{^{-1,1^{*}}}(\R^{n})} + c\,\norm{\mathscr{A}[P]}_{{\dot{\sobo}}{^{k-1,1^{*}}}(\R^{n})}.
\end{align}
We recall from $\eqref{eq:ellipticitycancellingreduced}_{1}$ that $\mathbb{B}:\hold^\infty_c(\R^n;\ker\mathscr{A})\to \hold^\infty_c(\R^n;W)$ is elliptic. Hence, by \cite[Prop.\@ 4.2]{VS} there exists $\ell\in\mathbb{N}$ and a finite-dimensional vector space $F$ and an $\ell$-th order, homogeneous, linear, constant coefficient differential operator $\mathbb{L}:\hold^\infty_c(\R^n;W)\to \hold^\infty_c(\R^n;F)$ such that
  \begin{subequations}
 \begin{equation}\label{eq:kernel}
  \ker(\mathbb{L}[\xi])=\mathbb{B}[\xi](\ker\mathscr{A}) \qquad \text{for all $\xi\in\R^n\backslash\{0\}$}.
 \end{equation}
In particular, this implies that 
\begin{equation}\label{eq:cocancelling}
 \mathbb{L}\mathbb{B}\Pi_{\ker\mathscr{A}}P=0.
\end{equation}
\end{subequations}
Moreover, by our assumption \eqref{eq:ellipticitycancellingreduced} of $\mathbb{B}$ being reduced cancelling relative to $\mathscr{A}$, we have that $\mathbb{B}:\hold^\infty_c(\R^n;\ker\mathscr{A})\to \hold^\infty_c(\R^n;W)$ is cancelling, so 
\begin{equation}
 \bigcap_{\xi\in\R^n\backslash\{0\}}\mathbb{B}[\xi](\ker\mathscr{A})=\{0\}.
\end{equation}
Recalling \eqref{eq:kernel}, we thus obtain
\begin{equation}
 \bigcap_{\xi\in\R^n\backslash\{0\}}\ker(\mathbb{L}[\xi])=\{0\},
\end{equation}
meaning that $\mathbb{L}:\hold^\infty_c(\R^n;W)\to \hold^\infty_c(\R^n;F)$ is cocancelling. For $P\in\hold^\infty_c(\R^n;V)$ we have $\mathbb{B} P\in\hold^\infty_c(\R^n;W)$, and an application of the cocancelling operator $\mathbb{L}$ gives us:
\begin{align}
 \norm{\mathbb{B}P}_{\dot{\sobo}{}^{-1,1^*}(\R^n)}\quad &\overset{\mathclap{\text{Lem. \ref{lem:strongBBcocancelling}}}}{\le} \quad c\left(\norm{\mathbb{L}\mathbb{B}P}_{\dot{\sobo}{}^{-1-\ell,1^*}(\R^n)}+\norm{\mathbb{B}P}_{\lebe^1(\R^n)}\right) \notag \\
 & \!\!\overset{\eqref{eq:cocancelling}}{=} c\left( \norm{\mathbb{L}\mathbb{B}\Pi_{(\ker \mathscr{A})^\perp}P}_{\dot{\sobo}{}^{-1-\ell,1^*}(\R^n)}+\norm{\mathbb{B}P}_{\lebe^1(\R^n)}\right) \notag \\
 &\le \quad c\left(\norm{\mathbb{B}\Pi_{(\ker \mathscr{A})^\perp}P}_{\dot{\sobo}{}^{-1,1^*}(\R^n)}+\norm{\mathbb{B}P}_{\lebe^1(\R^n)}\right) \notag \\
 &\overset{\mathclap{\text{$\mathbb{B}$ $k$-th ord.}}}{\le} \quad c\left( \norm{\Pi_{(\ker \mathscr{A})^\perp}P}_{{\dot{\sobo}}{^{k-1,1^*}}(\R^n)}+\norm{\mathbb{B}P}_{\lebe^1(\R^n)}\right)\notag\\
 &\!\!\overset{\eqref{eq:injectivity}}{\le} \quad c\left(\norm{\mathscr{A}[P]}_{{\dot{\sobo}}{^{k-1,1^*}}(\R^n)}+\norm{\mathbb{B}P}_{\lebe^1(\R^n)}\right). \label{eq:p=1}
\end{align}
Gathering estimates, we finally arrive at 
\begin{align*}
\norm{P}_{{\dot{\sobo}}{^{k-1,1^{*}}}(\R^{n})} & \leq  \norm{\Pi_{\ker\mathscr{A}}[P]}_{{\dot{\sobo}}{^{k-1,1^{*}}}(\R^{n})} + \norm{\Pi_{(\ker\mathscr{A})^{\bot}}[P]}_{{\dot{\sobo}}{^{k-1,1^{*}}}(\R^{n})}\\ 
& \!\stackrel{\eqref{eq:injectivity}}{\leq} \norm{\Pi_{\ker\mathscr{A}}[P]}_{{\dot{\sobo}}{^{k-1,1^{*}}}(\R^{n})} + c\,\norm{\mathscr{A}[P]}_{{\dot{\sobo}}{^{k-1,1^{*}}}(\R^{n})}\\ 
& \!\!\!\!\!\!\stackrel{\eqref{eq:littlehelper},\,\eqref{eq:p=1}}{\leq}  c\,\left(\norm{\mathscr{A}[P]}_{{\dot{\sobo}}{^{k-1,1^*}}(\R^n)}+\norm{\mathbb{B}P}_{\lebe^1(\R^n)}\right).
\end{align*}
This gives us the desired estimate \eqref{eq:thm:L1KMS1}. 

For the converse implication  `\ref{item:L1KMSA1} $\Rightarrow$ \ref{item:L1KMSA2}' we have to show that the restricted operator $\mathbb{B}:\hold^\infty_c(\R^n;\ker\mathscr{A})\to \hold^\infty_c(\R^n;W)$ is elliptic and cancelling. An application of \eqref{eq:thm:L1KMS1} to maps $P\in \hold^\infty_c(\R^n;\ker\mathscr{A})$ gives us
\begin{equation}
 \norm{P}_{{\dot{\sobo}{}^{k-1,1^*}}(\R^n)}\leq c\,\norm{\mathbb{B} P}_{\lebe^1(\R^{n})},
\end{equation}
so that the requisite ellipticity and cancellation, and thus \ref{item:L1KMSA2}, follow from \cite[Thm.\@ 1.3]{VS}. This completes the proof.
\end{proof}
We briefly pause to comment on the particular strategy in the above proof:
\begin{remark}\label{rem:mainpoints}
Note that it is the estimate \eqref{eq:p=1} where we need the additional condition of reduced cancellation in the borderline case $p=1$, since for $1<p<n$ we have  $\norm{\mathbb{B}P}_{\dot{\sobo}{}^{-1,p^*}(\R^n)}\le c  \norm{\mathbb{B}P}_{\lebe^p(\R^n)}$ by the classical Sobolev embedding. Moreover, if we assumed $\mathbb{B}$ to be {elliptic and} cancelling (and not only relative to $\mathscr{A}$ in the sense of \eqref{eq:ellipticitycancellingreduced}), then estimate \eqref{eq:p=1} would immediately simplify to $\norm{\mathbb{B}P}_{{\dot{\sobo}}{^{-1,1^{*}}}(\R^{n})}\leq c\norm{\mathbb{B}P}_{\lebe^{1}(\R^{n})}$. 

In the sole presence of \eqref{eq:ellipticitycancellingreduced}, a suitably modified  argument still does not seem good enough to avoid strong Bourgain-Brezis estimates: If one replaces the first estimate in \eqref{eq:p=1} for a first order differential operator $\mathbb{B}$ by 
\begin{align}\label{eq:BBexplain}
\norm{\mathbb{B}P}_{{\dot{\sobo}}{}^{-1,1^{*}}(\R^{n})} \leq \norm{\mathbb{B}\Pi_{(\ker\mathscr{A})^{\bot}}P}_{{\dot{\sobo}}{}^{-1,1^{*}}(\R^{n})}+\norm{\mathbb{B}\Pi_{\ker\mathscr{A}}P}_{{\dot{\sobo}}{}^{-1,1^{*}}(\R^{n})}, 
\end{align} 
then the first term on the right-hand side can be bounded against $c\norm{\mathscr{A}[P]}_{\lebe^{1^{*}}(\R^{n})}$. If we then directly work from the Sobolev estimate \eqref{eq:Sobolev}, the second term on the right-hand side of \eqref{eq:BBexplain} is bounded by $c\norm{\mathbb{B}\Pi_{\ker\mathscr{A}}P}_{\lebe^{1}(\R^{n})}$. In order to re-introduce $\norm{\mathbb{B}P}_{\lebe^{1}(\R^{n})}$ in view of \eqref{eq:thm:L1KMS1}, we have to control  $\|\mathbb{B}\Pi_{(\ker\mathscr{A})^{\bot}}P\|_{\lebe^{1}(\R^{n})}$ in terms of $\|\mathscr{A}[P]\|_{\lebe^{1^{*}}(\R^{n})}$ and $\norm{\mathbb{B}P}_{\lebe^{1}(\R^{n})}$ exclusively. This however seems difficult, if not impossible, as the full $\lebe^{1}$-norm of $\mathbb{B}\Pi_{(\ker\mathscr{A})^{\bot}}P$ does not give us enough flexibility to get back to the $\lebe^{1^{*}}$-norm of $\mathscr{A}[P]$. Based on strong Bourgain-Brezis estimates though, estimate \eqref{eq:p=1} shows that the critical lower order terms \emph{can} be handled when being measured in negative Sobolev norms. In this sense, the use of strong Bourgain-Brezis estimates gives us the requisite flexibility to re-introduce the lower order term $\mathscr{A}[P]$.
\end{remark}
As a direct consequence of Theorem \ref{thm:L1KMS} and its proof, one obtains the following variant for \emph{partially cancelling} operators in the sense of \cite[\S 7.1]{VS}:
\begin{corollary}
 Under the conditions of Theorem \ref{thm:L1KMS} let $\mathbb{B}$ be reduced elliptic relative to $\mathscr{A}$ and $\mathcal{T}:W\to\widetilde{W}$ be linear. Then \eqref{eq:thm:L1KMS1} holds true for all $P\in \hold^\infty_c(\R^n;V)$ such that $\mathcal{T}(\mathbb{B}P)=0$ if and only if
 \begin{equation}
  \bigcap_{\xi\in\R^{n}\setminus\{0\}}\mathbb{B}[\xi](\ker\mathscr{A})\cap\ker\mathcal{T}=\{0\}.
 \end{equation}
\end{corollary}
Finally, following \cite{LMN}, Theorem \ref{thm:L1KMS} can be reformulated in generalised Sobolev spaces: 
\begin{corollary}\label{cor:affirmative}
Let $\Omega\subset\R^{n}$ be open and bounded, and suppose that $\mathscr{A}$ and $\mathbb{B}$ are as in Theorem \ref{thm:L1KMS}\ref{item:L1KMSA2}. If we define ${\sobo}{_{0}^{\mathscr{A},\mathbb{B},1^{*},1}}(\Omega)$ to be the closure of $\hold_{c}^{\infty}(\Omega;V)$ with respect to the norm 
\begin{align*}
\norm{P}_{\sobo{^{\mathscr{A},\mathbb{B},1^{*},1}}(\Omega)}:=\norm{\mathscr{A}[P]}_{{\sobo}{^{k-1,1^*}}(\Omega)}+\norm{\mathbb{B} P}_{\lebe^1(\Omega)}, 
\end{align*}
then we have that ${\sobo}{_{0}^{\mathscr{A},\mathbb{B},1^{*},1}}(\Omega)\hookrightarrow\sobo{^{k-1,1^{*}}}(\Omega;V)$.
\end{corollary}
\begin{proof}
Extending $P\in\hold_{c}^{\infty}(\Omega;V)$ by zero to the entire $\R^{n}$, Theorem \ref{thm:L1KMS} gives us the requisite inequality for the extended maps. The claimed embedding then directly follows from the definition of ${\sobo}{_{0}^{\mathscr{A},\mathbb{B},1^{*},1}}(\Omega)$ as a closure.
\end{proof}

\subsection{Other space scales and connections to strong Bourgain-Brezis-type estimates}

As discussed in Remark \ref{rem:mainpoints}, the strong Bourgain-Brezis estimates are crucial in the proof of Theorem \ref{thm:L1KMS}. If we aim for KMS-type inequalities involving different space scales, so e.g. fractional Sobolev spaces, it is not known at present whether such strong Bourgain-Brezis estimates hold in the requisite form; also see the discussion in \cite[\S 9]{VS}. Here one has the following result, which works subject to the (full) cancellation assumption on $\mathbb{B}$: 
\begin{corollary}\label{cor:KMS}
In the situation of Theorem \ref{thm:L1KMS}, let $\mathbb{B}$ be a constant rank operator which is reduced elliptic relative to $\mathscr{A}$  and \emph{(fully) cancelling}. In particular, $\mathbb{B}$ satisfies $\bigcap_{\xi\neq 0}\mathbb{B}[\xi](V)=\{0\}$. Let $0<s<1$ and $1<p<\infty$ be such that $\frac{1}{p}-\frac{s}{n}=1-\frac{1}{n}$. Then for any $1<q<\infty$ there exists a constant $c=c(\mathscr{A},\mathbb{B},s,q)>0$ such that we have 
\begin{align}\label{eq:fractional}
&\norm{\D^{k-1}P}_{{{\dot{\besov}}{_{p,q}^{s}}(\R^{n})}}\leq c\Big(\norm{\D^{k-1}\mathscr{A}[P]}_{{\dot{\besov}}{_{p,q}^{s}}(\R^{n})} + \norm{\mathbb{B}P}_{\lebe^{1}(\R^{n})} \Big),\\ \label{eq:fractional2}
&\norm{\D^{k-1}P}_{{{\dot{\mathrm{F}}}{_{p,q}^{s}}(\R^{n})}}\leq c\Big(\norm{\D^{k-1}\mathscr{A}[P]}_{{\dot{\mathrm{F}}}{_{p,q}^{s}}(\R^{n})} + \norm{\mathbb{B}P}_{\lebe^{1}(\R^{n})} \Big),
\end{align}
for all $P\in\hold_{c}^{\infty}(\R^{n};V)$. 
\end{corollary} 
\begin{proof} 
We only give the modifications in the proof of Theorem \ref{thm:L1KMS}, and focus on \eqref{eq:fractional}. We keep everything unchanged until \eqref{eq:injectivity}, but now estimate \eqref{eq:littlehelper} is replaced by 
\begin{align*}
\norm{\partial^{\alpha}\Pi_{\ker\mathscr{A}}[P]}_{{\dot{\besov}}{_{p,q}^{s}}(\R^{n})}& \leq c\big(\norm{\mathbb{B}P}_{{\dot{\besov}}{_{p,q}^{s-1}}(\R^{n})} + \norm{\D^{k-1}\mathscr{A}[P]}_{{\dot{\besov}}{_{p,q}^{s}}(\R^{n})}\big).
\end{align*}
As the reader might notice, the derivation of this estimate only requires the reduced ellipticity of $\mathbb{B}$ relative to $\mathscr{A}$. Since $\mathbb{B}$ has constant rank, \cite[Rem.\@ 1]{Raita} provides us with a linear, homogeneous, constant-coefficient differential operator $\mathbb{L}$ from $W$ to a finite dimensional euclidean space $F$ such that we have  
\begin{align}\label{eq:exactnessconstantrank}
\ker(\mathbb{L}[\xi])=\mathbb{B}[\xi](V)\qquad\text{for all}\;\xi\in\R^{n}\setminus\{0\}.
\end{align} 
By the full cancellation condition on $\mathbb{B}$, we then conclude from \eqref{eq:exactnessconstantrank} that $\mathbb{L}$ is (fully) cocancelling, so $\bigcap_{\xi\in\R^{n}\setminus\{0\}}\ker(\mathbb{L}[\xi])=\{0\}$. Therefore, \cite[Prop.\@ 8.8]{VS} gives us the estimate
\begin{align}\label{eq:cancelest}
\norm{\mathbb{B}P}_{{\dot{\besov}}{_{p,q}^{s-1}}(\R^{n})} \leq c\norm{\mathbb{B}P}_{\lebe^{1}(\R^{n})}
\end{align}
as a substitute for \eqref{eq:p=1}. To arrive at \eqref{eq:fractional}, we then may conclude the proof of \eqref{eq:fractional} as above for Theorem \ref{thm:L1KMS}. Inequality \eqref{eq:fractional2} is then established analogously, now resorting to \cite[Prop.\@ 8.7]{VS} instead of \cite[Prop.\@ 8.8]{VS}. The proof is complete.
\end{proof}
As is the case for Theorem \ref{thm:L1KMS}, Corollary \ref{cor:KMS} lets us retrieve and extend several results from \cite{GLN1,GmSp}; see Section \ref{sec:ex00} for a discussion. We now briefly compare the underlying assumptions of Theorem \ref{thm:L1KMS} and Corollary \ref{cor:KMS}:
\begin{remark}[On the constant rank hypothesis]
In Theorem \ref{thm:L1KMS}, we do \emph{not} require $\mathbb{B}$ to have  constant rank, and mere injectivity of $\mathbb{B}[\xi]\colon \ker\mathscr{A}\to W$ for all $\xi\in\R^{n}\setminus\{0\}$ is sufficient. In view of the proof, this is so because we only need the operator $\mathbb{L}$ to satisfy \eqref{eq:kernel}. Since $\mathbb{B}$ is assumed elliptic as an operator on the $\ker\mathscr{A}$-valued maps, the existence of such an operator directly comes out as a consequence of \cite[Prop.\@ 4.2]{VS}. In the framework of Corollary \ref{cor:KMS}, the present lack of the strong Bourgain-Brezis estimates for fractional norms forces us to work with fully cancelling operators. However, then we cannot resort to \cite[Prop.\@ 4.2]{VS} in order to obtain the requisite operator $\mathbb{L}$ because $\mathbb{B}$ is not necessarily elliptic on $\hold_{c}^{\infty}(\R^{n};V)$. The assumption of $\mathbb{B}$ having (global) constant rank then still lets us apply \cite[Rem.\@ 1]{Raita} and thereby conclude Corollary \ref{cor:KMS}.  In applications, however, the constant rank hypothesis is satisfied by most of the relevant operators, and specifically lets us retrieve known critical KMS-inequalities in Besov and Triebel-Lizorkin spaces, cf. \cite[\S 2.2]{GmSp}.
\end{remark}
Comparing the proofs of Theorem \ref{thm:L1KMS} and Corollary \ref{cor:KMS} then leads to the following larger picture: If a space supports a suitable version of strong Bourgain-Brezis estimates, then the corresponding limiting $\lebe^{1}$-KMS-inequalities hold subject to reduced ellipticity and cancellation. If only variants of estimates \eqref{eq:cancelest} instead of strong Bourgain-Brezis estimates are available, then constant rank, reduced ellipticity and full cancellation are required. This leads to a variety of inequalities which we do not record here explicitly; yet, it is not clear to us how to avoid strong Bourgain-Brezis estimates in order to give fully analogous proofs in both scenarios.

\section{Examples}\label{sec:examples}
In this section, we present several examples of operators which illustrate the strength of Theorem \ref{thm:L1KMS} in the physically relevant cases  $n\in\{2,3\}$. This lets us retrieve several previously established results in a unified manner, but also yields novel inequalities which could not be treated by the available methods so far. Specifically, this concerns the operators $(\dev)\sym\Curl$ from the introduction, for which we answer a borderline case having been left open by \textsc{M\"{u}ller} et al. \cite[Rem.\@ 3.6]{LMN}; cf. Section \ref{sec:ex1} below. Our general findings are concisely summarised in Figure \ref{fig:maintable}.

\subsection{Previously known KMS-inequalities as special cases}\label{sec:L1Curl}
If $V=\widetilde{V}=\R^{n\times n}$ and $\mathbb{B}=\Curl $, validity of the KMS-inequality 
\begin{align*}
\norm{P}_{\lebe^{\frac{n}{n-1}}(\R^{n})} \leq c\Big( \norm{\mathscr{A}[P]}_{\lebe^{\frac{n}{n-1}}(\R^{n})}+\norm{\Curl P}_{\lebe^{1}(\R^{n})}\Big),\qquad P\in\hold_{c}^{\infty}(\R^{n};\R^{n\times n})
\end{align*}
is known to be equivalent to (i) $\mathbb{A}u \coloneqq\mathscr{A}[\D u]$ being an elliptic operator \cite{GLN1,GmSp} if $n\geq 3$ and (ii) $\mathbb{A}u \coloneqq\mathscr{A}[\D u]$ being a $\mathbb{C}$-elliptic operator \cite{GLN1} if $n=2$. This now is a consequence of Theorem \ref{thm:L1KMS}: Since $\ker\mathbb{B}[\xi]= \R^{n}\otimes\xi$, the reduced ellipticity relative to $\mathscr{A}$ is equivalent to $\mathscr{A}[v\otimes\xi]=0$ implying $v=0$ whenever $\xi\in\R^{n}\setminus\{0\}$, and this is nothing but the ellipticity of $\mathbb{A}u\coloneqq\mathscr{A}[\D u]$. In the situation of (i), $\Curl$ is cancelling, and so this case directly follows from Theorem \ref{thm:L1KMS}. In the situation of (ii), i.e. in two dimensions, the $\Curl$ is not cancelling. However, by Lemma \ref{lem:Celliptic} we see that this first order differential operator is reduced $\mathbb{C}$-elliptic relative to $\mathscr{A}$ and as before conclude the $\mathbb{C}$-ellipticity of $\mathbb{A}u\coloneqq\mathscr{A}[\D u]$.

\subsection{Inequalities involving $(\dev)\sym\Curl$}\label{sec:ex1}
In this section, we present two important instances of operators which are of particular relevance in applications, yet could not be treated by the available methods so far. These operators are given by the symmetrised curl and the deviatoric symmetrised curl, which already have been addressed in the introduction. The requisite sharp limiting estimates, which are new in this context, are now a direct consequence of Theorem \ref{thm:L1KMS}.

In three dimensions the classical curl builds on the classical cross product, so that the matrix $\Curl$ of a $(3\times3)$-matrix field is in turn a $(3\times3)$-matrix field. From the symbolic point of view the matrix $\Curl$ is seen here as a multiplication with the following special skew-symmetric matrix from the right
\begin{align}
\label{eq_daichohweis3po0uozee3eiX}
 \Anti(-\xi)\coloneqq \begin{pmatrix} 0 & \xi_3 & -\xi_2 \\ -\xi_3 & 0 & \xi_1 \\ \xi_2 & -\xi_1 & 0 \end{pmatrix}, \quad \text{for } \xi\in\R^3.
\end{align}
Similar constructions are applicable in all dimensions $n\ge2$ but the generalised matrix $\Curl$ is in general not a square matrix field,  see e.g. \cite{Lew}.

\begin{example}[Symmetrised matrix curl]\label{ex:symcurl}
The matrix differential operator $\widehat{\mathbb{B}}=\sym\Curl$, acting row-wisely on $\R^{3\times 3}$-valued fields, is cancelling. To see this, we consider for $\xi\in\R^{3}\setminus\{0\}$ and $P\in\R^{3\times 3}$, the corresponding symbol map
\begin{equation}
\widehat{\mathbb{B}}[\xi]P = \sym(P\Anti\xi)=\frac12P\Anti\xi-\frac12(\Anti\xi)P^{\top}. 
\end{equation}
Let $\E\in\bigcap_{\xi\in\R^3\backslash\{0\}}\widehat{\mathbb{B}}[\xi](\R^{3\times3})$, so that $\E\in\R^{3\times3}_{\sym}$. For an arbitrary  $z\in\R^3\backslash\{0\}$, consider $\xi=z$. Then we have for all $P\in\R^{3\times3}$:
\begin{equation*}
 z^\top(\widehat{\mathbb{B}}[z]P)z= \frac12 z^\top P \underset{=z\times z}{\underbrace{\left(\Anti z\right) z}} -\frac12 \underset{=-(z\times z)^\top}{\underbrace{z^\top(\Anti z)}}P^\top z= 0
\end{equation*}
and, in particular, $z^\top\E z =0$ for all $z\in\R^3$. Since moreover $\E$ is symmetric, we conclude that $\E=0$, and this means that $\sym\Curl$ is cancelling. We include the latter argument for the convenience of the reader: Indeed, for all $w,z\in\R^3$ we have 
\begin{align*}
 0&=(w+z)^\top\E(w+z)= w^\top\E w + z^\top \E w + w^\top \E z + z^\top\E z \\ & = \underset{\in\R}{\underbrace{z^\top\E w}} + w^\top \E z \overset{\E\in\R^{3\times3}_{\sym}}{=} 2 w^\top\E z, 
\end{align*}
and by arbitrariness of $w$ and $z$ we deduce that $\E = 0$. 
\end{example}
 The following example also gives a positive answer to a borderline case left open in \cite{LMN}, which in fact requires the sharp conditions from Theorem \ref{thm:L1KMS}\ref{item:L1KMSA2}. 
 Different from the operator $\sym\Curl$ as considered in Example \ref{ex:symcurl}, the operator $\dev\sym\Curl$ is not cancelling. Still, it proves to be reduced cancelling with respect to certain part maps $\mathscr{A}$:
\begin{example}[Deviatoric symmetrised matrix curl]\label{ex:devsymCurl}
First we show that the differential operator $\overline{\mathbb{B}}=\dev\sym\Curl$ is \emph{not cancelling}. To this end, we consider the action of its corresponding symbol map:
 \begin{equation}
 \begin{split}
\overline{\mathbb{B}}[\xi]P &= \dev\sym(P\Anti\xi) \\ & =\tfrac12P\Anti\xi-\tfrac12(\Anti\xi)P^\top+\tfrac13\skalarProd{\skew P}{\Anti\xi}\,\bbone_3.
\end{split}
\end{equation}
We now establish that $Q\in\bigcap_{\xi\in\R^3\backslash\{0\}}\overline{\mathbb{B}}[\xi](\R^{3\times3})$ for every symmetric, trace-free matrix $Q\in\R^{3\times 3}$. First, we have in view of \eqref{eq_daichohweis3po0uozee3eiX} for every \(\xi, \zeta \in \R^3\),
\begin{equation*}
\Anti(\zeta)\Anti(\xi) = \xi \otimes \zeta - \skalarProd{\xi}{\eta}\,\bbone_3.
\end{equation*}
This is Lagrange's triple product expansion \(\xi \times (\zeta \times v) = \skalarProd{\xi}{v}\, \zeta - \skalarProd{\xi}{\zeta}\,v \). Hence we have for every \(\widetilde{Q} \in \R^{3 \times 3}\)
\begin{equation*}
  (\Anti( \widetilde{Q} \xi) - \widetilde{Q}\Anti(\xi))\Anti (\xi)
 = \xi \otimes (\widetilde{Q}\xi)   - \skalarProd{\xi}{\widetilde{Q}\,\xi}\,\bbone_3
  - \widetilde{Q} \xi \otimes \xi
  + \abs{\xi}^2 \widetilde{Q},
\end{equation*}
so that if \(\widetilde{Q}=Q\) is symmetric and trace-free, we have 
\begin{equation*}
 \dev\sym\big(\Anti (\xi) (\Anti( Q \xi) - \Anti(\xi)Q)\big) = \abs{\xi}^2 Q,
\end{equation*}
and thus \(Q \in \overline{\mathbb{B}}[\xi](\R^{3 \times 3})\).

Therefore, $\bigcap_{\xi\in\R^3\backslash\{0\}}\overline{\mathbb{B}}[\xi](\R^{3\times3})$ is non-trivial, meaning that $\dev\sym\Curl$ is \emph{not cancelling}. But, as we have seen in \cite[Sec.\@ 4.1.3]{GLN2}, the operator $\overline{\mathbb{B}}=\dev\sym\Curl$ is \emph{reduced} $\C$-elliptic relative to $\mathscr{A}=\sym$. Thus, by our Corollary \ref{cor:reducedCelliptic} it holds
\begin{equation}
 \norm{P}_{\lebe^{1^*}(\R^{3})}\leq c\,\left(\norm{\sym P}_{\lebe^{1^*}(\R^{3})}+\norm{\dev\sym\Curl P}_{\lebe^1(\R^{3})}\right) 
   \end{equation}
   for all $P\in\hold^\infty_c(\R^3;\R^{3\times3})$. Based on Corollary \ref{cor:affirmative}, we thus obtain the missing borderline case from \cite[Thm.\@ 3.5, Rem.\@ 3.6]{LMN} for globally vanishing traces:
   \begin{align*}
     \norm{P}_{\lebe^{1^*}(\Omega)}\leq c\,\left(\norm{\sym P}_{\lebe^{1^*}(\Omega)}+\norm{\dev\sym\Curl P}_{\lebe^1(\Omega)}\right) 
   \end{align*}
for all $P\in \sobo^{\,\sym,\dev\sym\Curl,1^*,1}_0(\Omega)$.
\end{example}

\subsection{A non-reduced $\mathbb{C}$-elliptic, yet reduced cancelling operator}
With $V=\widetilde V=\R^{3\times 3}$, $W=\R^2$ consider $\mathscr{A}=\dev$ and $\widetilde{\mathbb{B}}$ given by
\begin{equation}
 \widetilde{\mathbb{B}}P\coloneqq \begin{pmatrix} \mathrm{div}\Div P -\partial_3\mathrm{div} P^3 \\ \partial_3\mathrm{div}P^3 \end{pmatrix} ,
\end{equation}
where $P^3$ is either the third column or the third row of $P$. On $\ker\mathscr{A}$-valued maps this reduces to the elliptic and cancelling but not $\C$-elliptic second order differential operator $\mathbb{B}_{2,3}$ from the counterexample 3.4 in \cite{GmRa1}, and by our Theorem \ref{thm:L1KMS} we obtain:
\begin{equation}
 \norm{P}_{\dot{\sobo}{}^{1,1^*}(\R^3)}\le c\, (\norm{\dev P}_{\dot{\sobo}{}^{1,1^*}(\R^3)}+\|\widetilde{\mathbb{B}}P\|_{\lebe^1(\R^3)}).
\end{equation}
In a similar way, we can construct examples build upon the operators $\mathbb{A}_{k,n}$ from the counterexample 3.4 in \cite{GmRa1}.

\newcommand{\redCell}{\cellcolor{lightgray!20}\checkmark}

\begin{figure}\centering
 \begin{tabular}{c|c|c|c|c|c|c|c|}
\diagbox[innerwidth=1.5cm]{$\mathscr{A}$}{$\mathscr{B}$}& $\mathrm{Id}$ & $\dev$ & $\sym$ & $\dev\sym$ & $\skew+\tr$ & $\skew$ & $\tr$\\\hline
$\mathrm{Id}$ & \redCell & \redCell& \redCell & \redCell & \redCell & \redCell & \redCell \\\hline 
$\dev$ & \redCell & \redCell & $\lightning$ & $\lightning$ & \redCell & \redCell & $\lightning$ \\\hline
$\sym$ & \redCell & \redCell & \redCell & \redCell &
\diagbox[innerwidth=1.5cm,dir=SW]{\ $\lightning$}{\checkmark\ }& $\lightning$ & $\lightning$ \\\hline
$\dev\sym$ & \redCell & \redCell & $\lightning$ & $\lightning$ &
\diagbox[innerwidth=1.5cm,dir=SW]{\ $\lightning$}{\checkmark\ } & $\lightning$ & $\lightning$ \\\hline
$\skew+\tr$ &
\diagbox[dir=SW]{\ \checkmark}{\checkmark} & \diagbox[dir=SW]{\ \checkmark}{\checkmark}
& $\lightning$ & $\lightning$ & $\lightning$ & $\lightning$ & $\lightning$ \\\hline
$\skew$ & $\lightning$ & $\lightning$ & $\lightning$ & $\lightning$ & $\lightning$ & $\lightning$ & $\lightning$ \\\hline 
$\tr$ & $\lightning$ & $\lightning$ & $\lightning$ & $\lightning$ & $\lightning$ & $\lightning$ & $\lightning$ \\\hline 
\end{tabular}
\caption{Overview of what type of KMS-inequalities apply, here \begin{tabular}{|c|}\redCell\end{tabular} marks reduced $\C$-ellipticity so that all KMS-inequalities are valid with this combination of $\mathscr{A}$ and $\mathbb{B}=\mathscr{B}[\Curl]$, in a 
bisected cell the upper part refers to the validity of $\lebe^1$-KMS inequalities and the lower part to the validity of $\lebe^p$-KMS ($p>1$) inequalities, whereby $\lightning$ denotes non-validity.} 
\label{fig:maintable}
\end{figure}

\subsection{Filling up Figure \ref{fig:maintable}}

The examples section in \cite[Sec.\@ 4]{GLN2} already provides a good overview of reduced ($\C$-)ellipticity of specific constellations for $\Curl$-based operators, to complete Figure \ref{fig:maintable} we only need to address the following constellations:

For $\mathscr{A}\in\{\sym,\dev\sym\}$ and $\mathscr{B}=\skew+\tr$ recall from \cite[Sec.\@ 4]{GLN2} that $\mathscr{B}[\Curl]$ is not reduced $\C$-elliptic.  On the other hand, $\mathscr{B}[\Curl]$ is reduced $\R$-elliptic so that the $\lebe^p$-KMS inequality holds for $p>1$ with this combination. We have already seen in Example \ref{ex:noninequality} above that the corresponding inequality does not persist in the borderline case $p=1$. Let us give the algebraic argument here. To this end we show, that $\mathscr{B}[\Curl]$ is not reduced cancelling relative to $\mathscr{A}$. Indeed, the action of the corresponding symbol map on skew-symmetric matrices gives us:
\begin{align}
 \skew(\Anti a\Anti\xi)+\tr(\Anti a\Anti\xi)\cdot\bbone_3= \frac12\Anti(a\times\xi)-2\skalarProd{a}{\xi}\,\bbone_3.
\end{align}
Hence,
\begin{align}
 \bigcap_{\xi\in\R^3\backslash\{0\}}\mathscr{B}[\Curl][\xi](\R^{3\times3}_{\skew})\supseteq\R\cdot\bbone_3
\end{align}
meaning that $\mathscr{B}[\Curl]$ is not reduced cancelling relative to $\mathscr{A}\in\{\sym,\dev\sym\}$.

\subsection{Remaining cases}\label{subsec:devCurl}
For $\mathscr{A}=\skew+\tr$ and $\mathscr{B}\in\{\mathrm{Id},\dev\}$ recall again from \cite[\S 4]{GLN2} that  $\mathscr{B}[\Curl]$ is not reduced $\C$-elliptic. But $\mathscr{B}[\Curl]$ is reduced $\R$-elliptic and the $\lebe^p$-KMS inequality holds for $p>1$ with this combination. Since $\Curl$ is a cancelling operator, this inequality also persists in the borderline case $p=1$ for $\mathscr{B}=\mathrm{Id}$. We show that the differential operator $\widetilde{\mathbb{B}}=\dev\Curl$ is also cancelling. For the corresponding symbol map we obtain
\begin{equation}
\widetilde{\mathbb{B}}[\xi]P=\dev(P\Anti\xi)=P\Anti\xi+\frac13\skalarProd{\skew P}{\Anti\xi}\,\bbone_3.
\end{equation}
Let $\E\in\bigcap_{\xi\in\R^3\backslash\{0\}}\widetilde{\mathbb{B}}[\xi](\R^{3\times3})$, then $\tr\E=0$, i.e., $\E\in\mathfrak{sl}(3)$. For all $w,z\in\R^3\backslash\{0\}$ with $w\perp z$ consider $\xi=z$. Then we have for all $P\in\R^{3\times3}$:
\begin{equation*}
  w^\top(\widetilde{\mathbb{B}}[z]P)z= w^\top P \underset{=z\times z}{\underbrace{\left(\Anti z\right) z}}+\frac13\skalarProd{\skew P}{\Anti z}\skalarProd{w}{z} = 0
\end{equation*}
and in particular $w^\top\E z =0$ for all $w,z\in\R^3$ with $w\perp z$, so that $\E$ has a diagonal structure. Since moreover $\E$ is trace-free, there exist $\alpha,\beta\in\R$ such that $\E=\operatorname{diag}(\alpha,\beta,-\alpha-\beta)$. We conclude:
\begin{align*}
 0 = \begin{pmatrix} 1 & -1 & 0 \end{pmatrix} \E \begin{pmatrix} 1 \\ 1 \\ 0 \end{pmatrix} = \alpha-\beta \quad \text{and} \quad 0 = \begin{pmatrix} 0 & 1 & -1 \end{pmatrix} \E \begin{pmatrix} 0 \\ 1 \\ 1 \end{pmatrix} = \alpha+2\beta 
\end{align*}
so that $\alpha=\beta=0$, i.e., $\E=0$, meaning that $\dev\Curl$ is cancelling and by our Theorem \ref{thm:L1KMS} we obtain the optimal $\lebe^1$-KMS inequality for all $P\in\hold^\infty_c(\R^3;\R^{3\times3})$:
\begin{align}
 \norm{P}_{\lebe^{1^*}(\R^3)}\le c\left(\norm{\skew P}_{\lebe^{1^*}(\R^3)}+\norm{\tr P}_{\lebe^{1^*}(\R^3)}+\norm{\dev\Curl P}_{\lebe^1(\R^3)}\right).
\end{align}

\subsection{Inequalities involving $\dev$-$\Div$}\label{sec:ex00}
 Let us also address shortly the optimal $\dev$-$\Div$-inequalities. In \cite[Sec.\@ 4.2]{GLN2} we have already seen that in all dimensions the matrix divergence operator is reduced $\C$-elliptic relative to the deviatoric part map, thus, by our Corollary \ref{cor:reducedCelliptic} we obtain the optimal estimate
\begin{equation}
 \norm{P}_{\lebe^{1^*}(\R^{n})}\leq c\,\left(\norm{\dev P}_{\lebe^{1^*}(\R^{n})}+\norm{\Div P}_{\lebe^1(\R^{n})}\right) 
   \end{equation}
   for all $P\in\hold^\infty_c(\R^n;\R^{n\times n})$.
 
\subsection{Kr\"{o}ner's incompatibility operator $\inc$}
As an example for higher order differential operators we consider $\inc$-based differential operators of the form $\mathbb{B}=\mathscr{B}[\inc]$, where $\inc=\Curl\circ\Curl^\top$ denotes \textsc{Kr\"oner}'s incompatibility tensor. We have shown in \cite{GLN2} that the only non-trivial constellation in which this type of operators could turn out to be reduced elliptic is for choosing $\mathscr{A}=\dev$. Moreover, for $\mathscr{B}\in\{\mathrm{Id}, \dev, \sym, \dev\sym\}$ the corresponding operators are even reduced $\C$-elliptic, see \cite{GLN2}. Thus, the generalised $\lebe^1$-KMS-inequalities hold for these combinations, and, e.g., we have
\begin{equation}
 \norm{P}_{\dot{\sobo}{}^{1,1^*}(\R^3)}\le c\,(\norm{\dev P}_{\dot{\sobo}{}^{1,1^*}(\R^3)}+\norm{\dev\sym\inc P}_{\lebe^1(\R^3)}).
\end{equation}
Recall that for $\mathscr{B}\in\{\skew+\tr,\tr\}$ the corresponding operator is $\R$-elliptic, so that the optimal KMS-inequalities hold in the cases $p>1$, see \cite{GLN2}. However, these inequalities do not persist in the borderline case $p=1$. Indeed, we have 
\[\tr\inc(\zeta\cdot\bbone_3)=2\,\Delta\zeta\cdot\bbone_3\quad\text{and}\quad\skew\inc(\zeta\cdot\bbone_3)\equiv0,\]
and the required reduced cancellation property is not fulfilled since the Laplacian is not cancelling.

 {\footnotesize
\subsection*{Acknowledgment}
The authors are grateful for possible discussions during the conference \emph{Trends in Analysis 2023} which was organized by the second author at the University of Duisburg-Essen in October 2023.
\vspace{-1.7ex}
\subsection*{Conflict of interest} The authors declare that they have no conflict of interest.
\vspace{-2.5ex}
\subsection*{Data availability statement} Data sharing not applicable to this article as no datasets were generated or analysed.}


\begin{thebibliography}{99}

\bibitem{Arnold} {Arnold, D.N.; Douglas, J.; Gupta, C.P.: \textit{A family of higher order mixed finite element methods for plane elasticity.} Numer. Math. \textbf{45} (1984), pp. 1--22.}

\bibitem{BauerNeffPaulyStarke} {Bauer, S., Neff, P., Pauly, D., Starke, G.: \textit{Dev-Div- and DevSym-DevCurl-inequalities for incompatible square tensor fields with mixed boundary conditions.} ESAIM Control Optim. Calc. Var. \textbf{22}(1) (2016), pp. 112--133.}

\bibitem{BB} {Bourgain, J.; Brezis, H.: \textit{New estimates for elliptic equations and Hodge type systems.} J. Eur. Math. Soc. (JEMS) \textbf{9}(2) (2007), pp. 277--315.}

\bibitem{BRWY} {Bousquet, P.; Russ, E.; Wang, Y.; Yung, P.-L.: \textit{Approximation in higher-order Sobolev spaces and Hodge systems.} J. Funct. Anal. \textbf{276}(5) (2019), pp. 1430--1478.}

\bibitem{BDG} {Breit, D.; Diening, L.; Gmeineder, F.: \textit{On the trace operator for functions of bounded $\mathbb{A}$-variation.} Analysis \& PDE \textbf{13}(2) (2020), pp. 559--594.}

\bibitem{ContiGarroni} {Conti, S.; Garroni, A.: \textit{Sharp rigidity estimates for incompatible fields as consequence of the Bourgain Brezis div-curl result.} C. R. Math. Acad. Sci. Paris \textbf{359}(2) (2021), pp. 155--160.}

\bibitem{Friedrichs} {Friedrichs, K.O.: \textit{On the boundary value problems of the theory of elasticity and Korn's inequality.} Ann. Math. \textbf{48}(2) (1947), pp. 441--471.}

\bibitem{GLP} {Garroni, A.; Leoni, G.; Ponsiglione, M.:  \textit{Gradient theory for plasticity via homogenization of discrete dislocations.} J. Eur. Math. Soc. (JEMS) \textbf{12}(5) (2010), pp. 1231--1266.}

\bibitem{GLN2} {Gmeineder, F.; Lewintan, P.; Neff, P.: \textit{Korn-Maxwell-Sobolev inequalities for general incompatibilities.} Math. Mod. Meth. Appl. Sci. ($\text{M}^{3}\text{AS}$) \textbf{34}(03) (2024), pp. 523--570.}

\bibitem{GLN1} {Gmeineder, F.; Lewintan, P.; Neff, P.: \textit{Optimal incompatible Korn-Maxwell-Sobolev inequalities in all dimensions.} Calc. Var. PDE \textbf{62} (2023), 182.}

\bibitem{GmRa1} {Gmeineder, F.; Rai\c{t}\u{a}, B.: \textit{Embeddings for $\mathbb{A}$-weakly differentiable functions on domains.} J. Funct. Anal. \textbf{277}(12) (2019), 108278.}

\bibitem{GmRaVa1} {Gmeineder, F.; Rai\c{t}\u{a}, B.; Van Schaftingen, J.: \textit{On limiting trace inequalities for vectorial differential operators.} Indiana Univ. Math. J. \textbf{70}(5) (2021), pp. 2133--2176.}

\bibitem{GmSp} {Gmeineder, F.; Spector, D.: \textit{On Korn-Maxwell-Sobolev inequalities.}  J. Math. Anal. Appl. \textbf{502}(1) (2021),125226.}

\bibitem{Gobert} Gobert, J.: \textit{Une in\'{e}galit\'{e} fondamentale de la th\'{e}orie de l'\'{e}lasticit\'{e}.} Bull. Soc. R. Sci. Li\`{e}ge \textbf{31} (1962), pp. 182--191.

\bibitem{Korn} Korn, A.: \textit{\"{U}ber einige Ungleichungen, welche in der Theorie der elastischen und elektrischen Schwingungen eine Rolle spielen.} Bulletin International de l'Acad\'{e}mie des Sciences de Cracovie, deuxi\`{e}me semestre \textbf{9}(37) (1909), pp. 705--724.

\bibitem{Lew} {Lewintan, P.: \textit{Matrix representation of a cross product and related curl-based differential operators in all space dimensions.} Open Mathematics \textbf{19}(1) (2021), pp. 1330--1348.}

\bibitem{LMN} {Lewintan, P.; M\"{u}ller, S.; Neff, P.: \textit{Korn inequalities for incompatible tensor fields in three space dimensions with conformally invariant dislocation energy.} Calc. Var. PDE \textbf{60} (2021), 150.}

\bibitem{Lewintan3} {Lewintan, P.; Neff, P.: \textit{$\lebe^{p}$-trace-free generalized Korn inequalities for incompatible tensor fields in three space dimensions.} Proc. Roy. Soc. Edinburgh Sect. A \textbf{152}(6) (2021), pp. 1477--1508.}

\bibitem{Lewintan4} {Lewintan, P.; Neff, P.: \textit{$\lebe^{p}$-versions of generalized Korn inequalities for incompatible tensor fields in arbitrary dimensions with $p$-integrable exterior derivative.} C. R. Math. Acad. Sci. Paris \textbf{359}(6) (2021), pp. 749--755.}

\bibitem{Nerf} {Neff, P., Ghiba, I.-D., Lazar, M., Madeo, A.: \textit{The relaxed linear micromorphic continuum: well-posedness of the static problem and relations to the gauge theory of dislocations.} Quart. J. Mech. Appl. Math. \textbf{68}(1) (2015), pp. 53--84.} 

\bibitem{NeffPlastic} {Neff, P.; Pauly, D.; Witsch, K.-J.: \textit{On a canonical extension of Korn's first and Poincar\'{e}'s inequality to $\mathsf{H}(\Curl)$.} J. Math. Sci. (N. Y.) \textbf{185}(5) (2012), pp. 721--727.}

\bibitem{NeffPaulyWitsch} {Neff, P.; Pauly, D.; Witsch, K.-J.: \textit{Poincar\'{e} meets Korn via Maxwell: Extending Korn's first inequality to incompatible tensor fields.} J. Differential Equations \textbf{258}(4) (2015), pp. 1267--1302.}

\bibitem{Ornstein} {Ornstein, D.: \textit{A non-inequality for differential operators in the $\lebe_{1}$ norm.} Arch. Ration. Mech. Anal. \textbf{11} (1962), pp. 40--49.}

\bibitem{Raita} {Rai\c{t}\u{a}, B.: \textit{Potentials for $\mathcal{A}$-quasiconvexity.} Calc. Var. PDE \textbf{58} (2019), 105.}

\bibitem{VS04} {Van Schaftingen, J.: \textit{Estimates for $\lebe^1$ vector fields.} C. R. Math. Acad. Sci. Paris \textbf{338}(1) (2004), pp. 23--26.}

\bibitem{VS} {Van Schaftingen, J.:  \textit{Limiting Sobolev inequalities for vector fields and canceling linear differential operators.} J. Eur. Math. Soc. (JEMS) \textbf{15}(3) (2013), pp. 877--921.}
\end{thebibliography}
\end{document}